\documentclass[10pt]{amsart}
\usepackage{amssymb}
\usepackage{epsfig}
\usepackage{url}
\usepackage{setspace}
\usepackage{color}
\theoremstyle{plain}

\newtheorem{thm}{Theorem}[section]
\newtheorem{cor}[thm]{Corollary}

\newtheorem{rem}[thm]{Remark}

\newtheorem{conj}[thm]{Conjecture}

\def\bbb{\mathbb}

\renewcommand{\phi}{\varphi}

\newcommand{\Z}{\bbb{Z}}
\newcommand{\Q}{\bbb{Q}}

\newcommand{\de}{\delta}
\newcommand{\ep}{\epsilon}
\newcommand{\tht}{\theta}

\begin{document}
\author{Andrew Bremner (Tempe)
\and Maciej Ulas (Krak\'ow)
}
\title{On the reducibility type of trinomials}
\maketitle
\begin{abstract}
Say a trinomial $x^n+A x^m+B \in \Q[x]$ has reducibility type $(n_1,n_2,...,n_k)$ if there exists a factorization of the trinomial
into irreducible polynomials in $\Q[x]$ of degrees $n_1$, $n_2$,...,$n_k$, ordered so that  $n_1 \leq n_2 \leq ... \leq n_k$.
Specifying the reducibility type of a monic polynomial of fixed degree is equivalent to specifying rational points on an 
algebraic curve. When the genus of this curve is 0 or 1, there is reasonable hope that all its rational points may be described;
and techniques are available that may also find all points when the genus is 2. Thus all corresponding reducibility types may be
described. These low genus instances are the ones studied in this paper.
\end{abstract}

\vspace{0.25in}


\section{Introduction}\label{sec0}
The reducibility over $\Q$ of trinomials $x^n+A x^m+B \in \Q[x]$, $A
B \neq 0$ (where without loss of generality $n \geq 2m$), has a long
history, and the reader is referred to Schinzel~\cite{Sch1}, reprinted in Schinzel~\cite{Sch4}, who
provides an excellent documentation and full bibliography.
Here, we study the {\it type} of reducibility of
trinomials. We say a trinomial has reducibility type
$(n_1,n_2,...,n_k)$ if there exists a factorization of the trinomial
into irreducible polynomials in $\Q[x]$ of degrees $n_1$, $n_2$,...,$n_k$.
Types are ordered so that $n_1 \leq n_2 \leq ... \leq n_k$. Thus,
for example,
\[ x^5-341x+780 = (x-3) (x^2+x+20)(x^2+2x-13) \]
has reducibility type $(1,2,2)$. We consider trinomials only up to scaling,
in that the polynomials $x^n+A x^m+B$ and $x^n + A \lambda^{n-m} x^m +B \lambda^n$, $\lambda \in \Q$,
have exactly the same factorization type; in particular, when the trinomial has a rational root,
we may assume that this root is $1$. We shall always assume $AB \neq 0$. \\
A monic polynomial of degree $d$ is determined by $d$ coefficients. Specify
the reducibility type of $x^n+Ax^m+B$ as $(d_1,d_2,...,d_k)$; then comparing coefficients
of $x^r$, $1 \leq r < n$, $r \neq m$, leads to $n-2$ equations involving $\sum_{i=1}^{i=k} d_i=n$ coefficients.
With the appropriate weighting of the coefficients, the variety so determined has
generic dimension $1$ in weighted projective space, so is a curve.  When the genus of this
curve is 0 or 1, there is reasonable hope that all its rational points may be described;
and techniques are available that may also find all points when the genus is 2. These
low genus instances are the ones we investigate in this paper. \\
Although our motivation is the reducibility
of trinomials over $\Q$ it is clear that the constructions of curves
related to a particular type of reducibility can be viewed over
other fields (of characteristic 0).\\
It is immediate to verify the following. The quadratic $x^2+A x+B \in
\Q[x]$ is reducible, so of type $(1,1)$, if and only if up to scaling
$(A,B)=(u-1,-u)$, for $u \in \Q\setminus\{0,1\}$.
A cubic $x^3+A x+B$ is of type $(1,2)$ if and only if up to scaling
$(A,B)=(u-1,-u)$, for $u \in \Q\setminus\{0,1\}$ and $-4u+1 \neq \Box$;
and is of type $(1,1,1)$ if and only if $(A,B)=(-v^2+v-1,v^2-v)$, for
$v \in \Q\setminus\{0,1\}$.\\
\section{Reducibility type of quartic trinomials}\label{sec1}
\begin{thm}\label{1111type}
The trinomial $x^4+A x+B$ has reducibility type $(1,1,2)$ if and only if
up to scaling
\begin{equation*}
A=-(u+1)(u^2+1),\quad B=u(u^2+u+1), \qquad u \in\Q\setminus\{0,-1\},
\end{equation*}
with factorization
\begin{equation*}
x^4+A x+B = (x-1)(x-u)(x^2+(u+1)x+(u^2+u+1)).
\end{equation*}
There are no polynomials $x^4+A x+B$ with reducibility type $(1,1,1,1)$.
\end{thm}

\begin{proof}
Suppose we have reducibility type $(1,1,2)$. By scaling, there is a linear factor $x-1$,
and $x^4+A x+B=(x-1)(x-u)(x^2+p x+q)$. On comparing coefficients of powers of $x$,
\[ p-u-1=0, \quad q-p u-p+u=0, \quad A=-q u-q+p u, \quad B=q u, \]
and thus
\[ p=u+1, \quad q=u^2+u+1, \quad (A,B)=(-(u+1)(u^2+1), u(u^2+u+1)), \]
as required.  The discriminant of the quadratic factor is $-3u^2-2 u-3<0$,
so that reducibility type $(1,1,1,1)$ is impossible.
\end{proof}

\begin{thm}
$x^4+A x^2+B$ has reducibility type $(1,1,2)$ if and only if up to scaling
\begin{equation*}
(A,B)=(q-1,-q), \qquad q\in\Q\setminus\{1\}, \qquad -q \neq \square;
\end{equation*}
with factorization
\begin{equation*}
x^4+A x^2+B=(x-1)(x+1)(x^2+q);
\end{equation*}
and has reducibility type $(1,1,1,1)$ if and only if
\begin{equation*}
(A,B)=(-u^2-1, u^2), \qquad u \in\Q\setminus\{0\},
\end{equation*}
with factorization
\[ x^4+A x^2+B=(x-1)(x+1)(x-u)(x+u). \]
\end{thm}

\begin{proof}
By scaling, suppose $x^4+A x^2+B=(x-1)(x-u)(x^2+p x+q)$; then equating coefficients of powers of $x$:
\[ p-u-1=0, \quad A=q-p u-p+u, \quad -q u-q+p u=0, \quad B=q u. \]
It follows that $p(q-u)=0$. If $p=0$ then $u=-1$ and $(A,B)=(q-1, -q)$;
and if $q=u$, then $(A,B)=(-u^2-1, u^2)$.
\end{proof}

\section{Reducibility type of quintic trinomials}

\begin{thm}\label{quintype1}
The trinomial $x^5+ A x+B$ has reducibility type $(1,2,2)$ if and only if up to scaling
\begin{align*}
&A=\frac{(v^2-v-1)(v^4-2v^3+4v^2-3v+1)}{(2v-1)^2}, \\
&B=-\frac{v(v-1)(v^2+1)(v^2-2v+2)}{(2v-1)^2}
\end{align*}
for $v\in\Q$ with $v \neq 0,1,1/2$. There are no trinomials $x^5+A x+B$
with reducibility type $(1,1,1,2)$.
\end{thm}

\begin{proof}
By scaling we suppose the linear factor $x-1$, with
\begin{equation*}
x^5+A x+B=(x-1)f_{1}(x)f_{2}(x)=(x-1)(x^2+p x+q)(x^2+v x+w).
\end{equation*}
Comparing coefficients of powers of $x$, we get the system of equations
\[ p+v-1=0, \quad (v-1)p+q-v+w=0, \quad (v-w)p-(v-1)q+w=0, \]
\[ A+p w+(v-w)q=0, \qquad B+q w=0. \]
Solving for $p,q,w,A,B$ we get $A,B$ as given in the statement of the Theorem
and
\[ p=1-v, \quad q=\frac{v(v^2-2v+2)}{2v-1}, \quad w=\frac{(v-1)(v^2+1)}{2v-1}. \]
It remains to show that the quadratics $f_i(x)$ are irreducible.
Without loss of generality, we show that $f_2(x)$ is irreducible.
The discriminant of $f_2(x)$ is $v^2-4w=-(2v^3-3v^2+4v-4)/(2v-1)$, so
$f_2(x)$ is reducible if and only if
\[ -(2v-1)(2v^3-3v^2+4v-4) = \Box. \]
This latter is the equation of an elliptic curve with minimal model
$y^2+x y+y=x^3-x-2$, and rank 0 over $\Q$. The torsion group is of
order 3, whose points lead to $A B=0$. Hence there is no non-trivial
specialization of $v$ which leads to the reducibility of $f_2(x)$, and the
Theorem follows.
\end{proof}

\begin{thm}\label{quintype2}
There are no trinomials $x^5+A x^2+B$ with reducibility type $(1,2,2)$ or
$(1,1,1,2)$.
\end{thm}
\begin{proof}
By scaling, suppose that $x^5+A x^2+B=(x-1)(x^2+p x+q)(x^2+r x+s)$. Equating coefficients of $x^4$, $x^3$, $x$:
\[  p+r-1=0, \quad q+p r+s-p-r=0, \quad q s-q r-p s=0. \]
On eliminating $r,s$, it follows that
\[ q^2 - q (p+p^2) + p-p^2+p^3 = 0, \]
whence
\[ (p+p^2)^2-4(p-p^2+p^3) = \Box, \]
the equation of an elliptic curve with minimal model $y^2+x y+y=x^3+x^2$,
of rank 0, and with torsion group $\Z/4\Z$, generated by $(x,y)=(0,0)$.
The four torsion points lead to $B=0$.
\end{proof}

\begin{cor}\label{cor1}
The equation (representing a curve of genus 3)
\begin{equation*}
G(p,q,r): (p+q)(q+r)(r+p)(p+q+r)-(p q+q r+r p)^2=0
\end{equation*}
has only trivial rational solutions i.e. those with $pqr=0$.
\end{cor}
\begin{proof}
Suppose $G(p,q,r)=0$ with $pqr \neq 0$. Certainly the absolute values of $p,q,r$ are distinct.
For if, without loss of generality, $p=q$, then $q(3q^3+6q^2 r+4q r^2+2r^3)=0$, so that $q=0$;
and if $p=-q$, then $q^4=0$, and again $q=0$.  Consider the trinomial
\begin{equation*}
H(X)=X^5-\frac{p^5-q^5}{p^2-q^2}X^2+\frac{p^2q^2(p^3-q^3)}{p^2-q^2}.
\end{equation*}
We have $H(p)=H(q)=0$ and $H(r)=(p-r)(q-r)G(p,q,r)/(p+q)=0$. This contradicts
Theorem \ref{quintype2}, since there are no trinomials $x^5+A x^2+B$ with three rational roots.
\end{proof}

\section{Reducibility type of sextic trinomials}

\begin{thm}\label{123type}
There are infinitely many trinomials $x^6+A x+B$ with reducibility type
$(1,2,3)$.
There are no trinomials $x^6+A x+B$ with reducibility type $(1,1,1,3)$.
\end{thm}

\begin{proof}
By scaling, suppose that $x^6+A x+B = (x-1)(x^2+v x+w)(x^3+p x^2+q
x+r)$. Equating coefficients of $x^5$, $\ldots$, $x^2$:
\[ p+v-1=0, \qquad q-p+p v-v+w=0, \qquad r-q+q v-p v+p w-w=0, \]
\[ -r+r v-q v+q w-p w=0; \]
and eliminating $p,q,r$:
\begin{equation}
\label{vweqn}
 (3v^2-2v+1-2w)^2 = (v-1)(5v^3-3v^2+3v+3).
\end{equation}
This latter is the equation of an elliptic curve $E$ with minimal model
$y^2=x^3+3x+1$, of rank 1 with generator $P(x,y)=(0,1)$.
Each multiple of $P$ pulls back to a trinomial of type $x^6+A x+B$
factoring into polynomials of degrees 1,2,3.
Now if the quadratic is reducible, then $v^2-4 w=\Box$, which
with (\ref{vweqn}) represents a curve of genus 3, so having only finitely
many points $v,w$.  The cubic factoring together with (\ref{vweqn})
represents a curve of genus 7 and so again, there are only finitely many
points $(v,w)$. Since the number of points on $E$ is infinite,
all but finitely many such points lead to trinomials with reducibility type $(1,2,3)$.
As an example, the point $2P=(9/4, -35/8)$ pulls back to the trinomial
\begin{equation*}
x^6-19656x+82655=(x-5)(x^2+13x+61)(x^3-8x^2+68x-271).
\end{equation*}
The second assertion of the Theorem follows from Theorem \ref{111n} below.
\end{proof}


\begin{thm}\label{222type}
There are no trinomials $x^6+A x+B$ with reducibility type $(1,1,2,2)$ or $(2,2,2)$.
\end{thm}
\begin{proof}
Suppose $x^6+A x+B = (x^2+p x+q)(x^2+r x+s)(x^2+u x+v)$.
Equating coefficients of $x^5$,...,$x^2$:
\begin{align*}
-p-r-u=0,& \qquad -p(u+r)-q-u r-s-v=0, \\
-(s+r u+v)p-(r+u)q-v r-u s=0,& \quad -(s u+r v)p-(s+r u+v)q-v s=0,
\end{align*}
and eliminating $p,q,s$ results in
\begin{align*}
3(r^2+r u+u^2) v^2 + & 2(r^4+2r^3 u-3r^2 u^2-4r u^3-2u^4) v \\
 & - (r^6+3r^5 u+5r^4 u^2 + 5r^3 u^3-2r u^5-u^6) = 0.
\end{align*}
The discriminant in $v$ must be a perfect square, leading to
\[ (2r+u)^2(r^6 + 3 r^5 u + 3 r^4 u^2 + r^3 u^3 + 3 r^2 u^4 + 3 r u^5 + u^6) = \Box. \]
The restriction $u=-2r$ leads to trivial solutions, and thus $(X,Z)=(u+2r,u)$ gives a point
on the genus 2 curve
\[ C: Y^2 = X^6 - 3 X^4 Z^2 + 51 X^2 Z^4 + 15 Z^6. \]
$C$ covers two obvious elliptic curves,
\[ E_1: y^2=x^3-3 x^2+51 x+15, \qquad E_2: y^2 = 15 x^3+51 x^2 -3 x+1, \]
but each $E_i$ is of rank 1 over $\Q$, so there is no simple way to compute
the finitely many rational points on $C$. However, we can argue as follows, and show that the only
rational points on $C$ are $(\pm X,\pm Y,Z)=(1,1,0)$, $(1,8,1)$.
Set $K={\Q}(\tht)$, $\tht^3+3\tht+1=0$. The ring of integers is $\Z[\tht]$; a fundamental unit
is $\ep=\tht$; and the class number is 1. We have factorizations into prime ideals as follows:
\[ (2), (3)=p_3^3, (5)=p_5^2 p_5', \]
with $p_3=(2+\tht^2)$, $p_5=(1+\tht^2)$, $p_5'=(4+\tht^2)$.
The equation becomes
\[ \mbox{Norm}_{K/{\Q}}(X^2 -(4\tht+1)Z^2) = Y^2, \]
equivalently,
\[ (X^2-(4\tht+1)Z^2) (X^2+(-2\tht^2-4)X Z+(2\tht^2+7)Z^2) (X^2+(2\tht^2+4)X Z+(2\tht^2+7)Z^2) = Y^2. \]
The gcd $(X^2 -(4\tht+1)Z^2, X^4+(4\tht-2)X^2 Z^2+(16\tht^2-4\tht+49)Z^4)$ divides $2^4 p_3^3 p_5$.
So
\[ (X^2 -(4\tht+1)Z^2) = 2^i p_3^j p_5^k \square, \qquad i,j,k=0,1; \]
and taking norms, $2^{3i} 3^j 5^k=\square$: so $i=j=k=0$.
Hence
\[ X^2 -(4\tht+1)Z^2 = \pm \epsilon^{-a} y_1^2, \qquad X^4+(4\tht-2)X^2 Z^2+(16\tht^2-4\tht+49)Z^4 = \pm \epsilon^a y_2^2, \]
for $a=0,-1$, and $y_1,y_2$ integers in $K$ with $y_1 y_2=y$.\\
Case of $+$ sign, $a=-1$:
\[ X^2 -(4\tht+1)Z^2 = \epsilon y_1^2, \qquad X^4+(4\tht-2)X^2 Z^2+(16\tht^2-4\tht+49)Z^4 = \epsilon^{-1} y_2^2, \]
and the latter is not locally solvable at $p_3$.\\
Case of $-$ sign, $a=0$:
\[ X^2 -(4\tht+1)Z^2 = - y_1^2, \qquad X^4+(4\tht-2)X^2 Z^2+(16\tht^2-4\tht+49)Z^4 = - y_2^2, \]
and the latter is not locally solvable at $p_3$.\\
Case of $+$ sign, $a=0$:
\[ X^2 -(4\tht+1)Z^2 = y_1^2, \qquad X^4+(4\tht-2)X^2 Z^2+(16\tht^2-4\tht+49)Z^4 = y_2^2, \]
and the latter is an elliptic curve of rank 1 over ${\Q}(\tht)$. The rank is smaller than $[K:\Q]$,
and the elliptic Chabauty routines in Magma~\cite{Mag} work effectively to show that the curve
has precisely one point with $X:Z$ rational, namely $(X,Z)=(1,0)$, with $(y_1,y_2)=(1,1)$.
This leads to $A=0$.\\
Case of $-$ sign, $a=-1$:
\[ X^2 -(4\tht+1)Z^2 = -\epsilon y_1^2, \qquad X^4+(4\tht-2)X^2 Z^2+(16\tht^2-4\tht+49)Z^4 = -\epsilon^{-1} y_2^2, \]
and the latter is an elliptic curve of rank 1 over ${\Q}(\tht)$. As
above, elliptic Chabauty techniques show there is precisely one pair
of points with $X:Z$ rational, namely $(X,\pm Z)=(1,1)$, with
$(y_1,y_2)=(2,4)$. These points pull back to $A=0$.
\end{proof}

\begin{thm}\label{222type} \hfill
\begin{enumerate}
\item There are no trinomials $x^6+A x^2+B$ with reducibility type $(1,2,3)$.
\item If $x^6+A x^2+B$ has reducibility type $(2,2,2)$ then either
\begin{equation*}
(A,B)=(-s^2-s v-v^2,\;-s v(s+v)), \qquad s,v \in \Q, \quad s v(s+v) \neq 0,
\end{equation*}
with
\[ x^6+A x^2+B = (x^2+s)(x^2+v)(x^2-s-v); \]
or, up to scaling,
\begin{equation*}
(A,B)=(-(v-1)(3v-1),\;-v^2(2v-1)), \quad v\in\Q, \; v \neq 0,1,1/2,1/3.
\end{equation*}
with
\[ x^6+A x^2+B = (x^2+x+v)(x^2-x+v)(x^2-2v+1). \]
\end{enumerate}
\end{thm}

\begin{proof}
The first part of the Theorem is a simple consequence of the fact
that if $f(u)=0$ for some $u\in\Q\setminus\{0\}$ then $f(-u)=0$ and
thus we have two rational roots. \\
In order to prove the second part of the Theorem, suppose
\[ x^6+A x^2+B = (x^2+px+q)(x^2+rx+s)(x^2+u x+v). \]
Equating coefficients of $x^5$, $x^4$, $x^3$, $x$:
\[ -p-r-u=0, \qquad -p(u+r)-q-u r-s-v=0, \]
\[ -(s+r u+v)p-(r+u)q-v r-u s=0, \qquad -q s u-q r v-p s v=0. \]
Eliminating $p,q,s$:
\[ r u(r + u)(r^4+2r^3 u+3r^2 u^2+2r u^3+2u^4-2r^2 v-2r u v-8u^2 v+9v^2)=0. \]
The latter factor has discriminant in $v$ equal to $-8(2r+u)^2(r^2+r u+u^2)$,
which is a perfect square if and only if $2r+u=0$, leading to $p=r=u=0$, $q=-s-v$, and
$(A,B)=(-s^2-s v-v^2,-s v(s+v))$.
If instead $r u=0$, then symmetry allows us to take without loss of generality $r=0$. Then
$p=-u$, $q=-s+u^2-v$, and $u(s-u^2+2v)=0$.
The case $u=0$ results in the previous factorization; and the case $s=u^2-2v$ results in
$(A,B)=(-(u^2-v)(u^2-3v), v^2(u^2-2v))$, which on scaling so that $u=1$, gives the assertion in the Theorem.
Finally, the case $r=-u$ leads to the same factorization.
\end{proof}

\begin{thm}\label{222type63} \hfill
\begin{enumerate}
 \item If $x^6+A x^3+B$ has reducibility type $(2,2,2)$ then up to scaling
\begin{equation*}
A=-27u(u+1), \qquad B=27(u^2+u+1)^3, \quad u \neq 0,-1,-1/2,
\end{equation*}
with
\begin{align*}
&x^6 - 27u(u+1) x^3 + 27(u^2+u+1)^3 = \\
&(x^2+3x+3(u^2+u+1)) (x^2-3(u+1)x+3(u^2+u+1)) (x^2+3u x+3(u^2+u+1)).
\end{align*}
\item If $x^6+A x^3+B$ has reducibility type $(1,1,2,2)$ then up to scaling
\begin{equation*}
A=-u^3-1,\quad B=u^3, \qquad u \neq 0,-1
\end{equation*}
and
\begin{equation*}
x^6+A x^3+B=(x-1)(x-u)(x^2+x+1)(x^2+u x+u^2).
\end{equation*}
\item There are no trinomials $x^6+A x^3+B$ with reducibility type $(1,1,1,1,2)$.
\item If $x^6+A x^3+B$ has reducibility type $(1,2,3)$, then up to scaling, $(A,B)=(t-1,-t)$,
where $t$ is not a cube, and
\[ x^6 + (t-1) x^3 - t = (x-1)(x^2+x+1)(x^3+t). \]
\end{enumerate}
\end{thm}

\begin{proof}
To prove the first two parts of the Theorem, suppose that
 $x^6+A x^3+B = (x^2+px+q)(x^2+rx+s)(x^2+ux+v)$ and equate
coefficients of $x^5$, $x^4$, $x^2$, $x$:
\[ p+r+u=0, \qquad p(u+r)+q+u r+s+v=0, \]
\[ q s+q r u+p s u+q v+p r v+s v=0, \qquad q s u+q r v+p s v=0. \]
Eliminating $r,s,v$:
\[ (p^2-q)(p^2+p u+u^2)(p^2-3q+p u+u^2)(q+p u+u^2)=0. \]
There are three cases:
\begin{enumerate}
\item If $q=p^2$, then either $v=u^2$ and on scaling so that $p=1$, we have the factorization
in statement (2) of the Theorem;
or $v=-p^2-p u$,
which leads to the same factorization under change of variable.
\item If $q=\frac{1}{3}(p^2+p u+u^2)$ then either $v=\frac{1}{3}(p^2+p u+u^2)$
and on scaling to $p=3$, we have the factorization in statement (1) of the Theorem;
or $v=-\frac{2}{3}p^2-\frac{5}{3}p u+\frac{1}{3}u^2$, $p \in \{0,u,-1/2 u,-2u\}$, with corresponding factorizations
that either have $A=0$ or are special cases of the factorization in case (1).
\item If $q=-p u-u^2$ then $v=u^2$ and
again we have a symmetry of the factorization in case (1).
\end{enumerate}
The third part of the Theorem is immediate from the second part.\\
Finally, suppose $x^6+A x^3+B=(x-u)(x^2+p x+q)(x^3+r x^2+s x+t)$, and equate
coefficients:
\[ p + r - u=0, \qquad q + p r + s - p u - r u=0, \]
\[ -A + q r + p s + t - q u - p r u - s u=0, \qquad q s + p t - q r u - p s u - t u, \]
\[ q t - q s u - p t u=0, \qquad -B - q t u=0. \]
If $r=0$, it follows that $p=u$, $s=0$, $q=u^2$, and we derive the factorization
in the statement (4) of the Theorem.
If $r \neq 0$, then eliminating $p,q,t$,
\[ (-s + r u) (r^4-2r^2 s+s^2-r^3 u+r s u+r^2 u^2)=0, \]
and the latter factor has discriminant in $u$ equal to $-3r^2(r^2-s)^2$, which is a perfect square
if and only if $r^2-s=0$. This leads to $B=0$.
And if $s=r u$, then $(A,B)=((r-2u)(r^2-r u+u^2),-(r-u)^3 u^3)$, and the factorization is
again of type $(1,1,2,2)$.
\end{proof}

\section{Reducibility type of some higher degree trinomials}

\begin{thm}
If $x^7+A x+B$ has reducibility type $(1,2,4)$ then
\begin{align*}
&A=w^3 - v(3v+1) w^2 + v(v^3+3v^2-2v+1)w - v(v^4-v^3+v^2-v+1),\\
&B=-w(w^2 - (3v^2-2v+1)w + (v^4-v^3+v^2-v+1))
\end{align*}
and
\begin{equation*}
4v^6-8v^5+9v^4-4v^3-6v^2+12v-3 = \square = (4v^3-3v^2+2v-1 -2(3v-1)w)^2
\end{equation*}
for some $v\in\Q$.
\end{thm}

\begin{proof}
We suppose $x^7+A x+B=(x-1)(x^2+v x+w)(x^4+p x^3+q x^2+r x+s)$, and equate coefficients:
\[ p + v - 1=0, \qquad -q + p - p v + v - w=0, \]
\[ -r +  q - q v + p v - p w + w=0, \qquad -s + r - r v + q v - q w + p w=0, \]
\[ s - s v + r v - r w + q w=0, \qquad A + s v - s w + r w=0, \qquad B + s w=0. \]
Eliminating $p,q,r,s$,
\[ A = w^3 - v(3v+1) w^2 + v(v^3+3v^2-2v+1)w - v(v^4-v^3+v^2-v+1), \]
\[ B= -w(w^2 - (3v^2-2v+1)w + (v^4-v^3+v^2-v+1)), \]
with
\[ 4v^6-8v^5+9v^4-4v^3-6v^2+12v-3 = (4v^3-3v^2+2v-1 -2(3v-1)w)^2. \]
\end{proof}

\noindent
We have been unable to determine the finitely many rational
points on the curve $C$ of genus 2:
\[ C:\;y^2=4v^6-8v^5+9v^4-4v^3-6v^2+12v-3. \]
but believe the set of (finite) points is the following:
\[ (v,\pm y)=(1,2), \quad (-1,2), \quad (1/3,14/27), \quad (7/5,446/125), \]
with corresponding $w=0$, $1$; $w=1$, $3/2$; $w=13/9$; and $w=13/25$, $327/200$, respectively.
The first leads to $B=0$; and after scaling, the others determine the following factorizations,
which we believe to be the complete list of the given type:
\begin{align*}
&x^7-232x+336=(x-2)(x^2-2x+6)(x^4+4x^3+6x^2-4x-28),\\
&x^7+1247x-5928=(x-3)(x^2+x+13)(x^4+2x^3-6x^2+7x+152),\\
&x^7-9073x-32760=(x-5)(x^2+7x+13)(x^4-2x^3+26x^2-31x+504),\\
&x^7-204214984x+2804299680=(x-20)(x^2+28x+654)(x^4-8x^3-30x^2+14072x-214396).
\end{align*}

\begin{rem}
{\rm It is interesting to note that to the best of our knowledge,
these trinomials with reducibility type $(1,2,4)$ give the first explicit
examples showing that some (exceptional) finite sets defined in Theorems 3 of
Schinzel~\cite{Sch2,Sch3} are non-empty.}
\end{rem}

\begin{thm}
\begin{enumerate}
\item If $x^7+A x+B$ has reducibility type $(3,4)$ then up to scaling
\begin{align*}
(A,B)= & ( (4u^2-5u+2)(u^3-u^2-2u+1)/(2u-3)^2, \\
       & (3u-1)(2u-1)(u-1)(u^2-u+1)/(2u-3)^2), \quad u \neq 1,\frac{1}{2}, \frac{1}{3}, \frac{3}{2},
\end{align*}
with
\begin{align*}
x^7+A x+B = & \left( x^3+x^2+u x+\frac{(3u-1)(u-1)}{2u-3} \right) \\
            & \left( x^4-x^3-(u-1)x^2+\frac{(u^2-4u+2)}{2u-3}x+\frac{(2u-1)(u^2-u+1)}{2u-3} \right).
\end{align*}
\item There are no trinomials $x^7+A x^2+B$ with reducibility type $(3,4)$.
\item If $x^7+A x^3+B$ has reducibility type $(3,4)$ then up to scaling $(A,B)=(2,-1)$ with
\[ x^7+A x^3+B = (x^3+x^2-1)(x^4-x^3+x^2+1). \]
\end{enumerate}
\end{thm}

\begin{proof}
The first part is immediate on comparing coefficients in the
expression $x^7+A x+B=(x^4+p x^3+q x^2+r x+s)(x^3+t x^2+u x+v)$, and
scaling so that $t=1$ (in fact this result can be found in
\cite{Sch1}). For the second part, suppose $x^7+A x^2+B=(x^4+p x^3+q
x^2+r x+s)(x^3+t x^2+u x+v)$. Comparing coefficients and eliminating
$p,q,r,s$:
\[ 4 u^3+4 t^2 u^2-4 t^4 u+t^6 = (t^3-4t u+2v)^2, \]
the equation of an elliptic curve with rank 0 and torsion group of order 3. The torsion leads
to the $(1,3,3)$ factorization $x^7-2x^2+1=(x-1)(x^3+x+1)(x^3+x^2-1)$. For the third part,
suppose $x^7+A x^3+B=(x^4+p x^3+q x^2+r x+s)(x^3+t x^2+u x+v)$. Comparing coefficients and eliminating
$p,q,r,s$:
\[ (u-t^2)(2u-t^2)(2u^2+t^2 u+t^4) = (2 t v-t^2 u-2 u^2+t^4)^2, \]
the equation of an elliptic curve of rank 0 and torsion group of order 6. The only non-trivial
trinomial that results is the one in the statement of the Theorem.
\end{proof}




\begin{thm}\label{deg8}
The trinomial $x^{8}+A x^3+B$ is divisible by the polynomial
$x^3+u x^2+v x+w$ if and only if
either $u=0$ and $(A,B)=(-3t^5,-t^8), \quad t \in \Q\setminus\{0\}$, with (upon scaling)
\[ x^8-3 x^3-1 = (x^3+x+1)(x^5-x^3-x^2+x-1); \]
or, upon scaling to $u=1$,
\begin{align}
&A=1-5v+6v^2-v^3+4w-6v w+w^2,\label{eq1}\\
&B=-w(v-3v^2+v^3-w+4v w-w^2),\label{eq2}
\end{align}
and
\begin{equation}
\label{eq3}
(v^2+2v-1)(4v^3-3v^2+2v-1)=\square = (2(v+2)w-(3v-1)(v+1))^2.
\end{equation}
In particular the only trinomials $x^{8}+Ax^3+B$ with reducibility type $(3,5)$ are those
with the numbers {\rm (1), (2), (3), (4)} on the list at the end of Schinzel~\cite{Sch1}.
\end{thm}

\begin{proof}
Applying the Division Algorithm, the remainder on dividing $x^8+A x^3+B$ by $x^3+u x^2+v x+w$
is equal to $a x^2+b x+c$, with
\begin{align*}
&a=-A u+u^6-5u^4 v+6u^2 v^2-v^3+4u^3 w-6u v w+w^2,\\
&b=-A v+u^5 v-4u^3 v^2+3u v^3-u^4 w+6u^2 v w-3v^2 w-2u w^2,\\
&c=B-A w+u^5 w-4u^3 v w+3u v^2 w+3u^2 w^2-2v w^2.
\end{align*}
If $u=0$, then it follows that $(v,w)=(t^2,t^3)$, $(A,B)=(-3t^5,-t^8)$, for some non-zero $t \in \Q$.
If $u \neq 0$, then setting $a=c=0$ gives (\ref{eq1}), (\ref{eq2}); and demanding $b=0$
gives (\ref{eq3}) in the form
\begin{equation*}
(v^2+2u^2v-u^4)(4v^3-3u^2v^2+2u^4v-u^6) = (2(v+2u^2)w-u(3v-u^2)(v+u^2))^2.
\end{equation*}
The problem of determining all trinomials $x^8+ax^3+b$ with reducibility type $(3,5)$ is
thus reduced to finding all rational points on the genus two curve
\begin{equation*}
C:\;(v^2+2u^2v-u^4)(4v^3-3u^2v^2+2u^4v-u^6)=\square,
\end{equation*}
where we can assume without loss of generality that $(v,u^2)=1$. We show
the points comprise precisely the following set: $(v,u^2)=(1,0),(0,1),(1,1),(-2,1)$,
which lead just to the factorizations on the Schinzel list.  \\
Observe first that $r=v/u^2$ satisfies $(r^2+2r-1)(4r^3-3r^2+2r-1)>0$, so necessarily
\begin{equation}
\label{realcond}
(-1-\sqrt{2}) < u/v^2 < (-1+\sqrt{2}), \mbox{ or } 0.60583 <  u/v^2.
\end{equation}
Factoring over $\Z$,
\begin{equation}
v^2+2u^2 v-u^4 = c_0 g^2, \quad 4v^3-3u^2 v^2+2u^4 v-u^6 = c_0 h^2,
\label{gheqs}
\end{equation}
with $c_0=\pm 1, \pm 2$, and $g,h \in \Z$. When $c_0=1,-2$, the latter elliptic curve has rank 0,
and leads only to $u=0$; so we need only consider $c_0=-1,2$.\\
We need to work over two number fields. First, $K=\Q(\sqrt{2})$ with integer ring $\Z[\sqrt{2}]$,
class number 1, and fundamental unit $e=1+\sqrt{2}$. Second, $L=\Q(\theta)$, where
$\theta^3-2\theta^2+3\theta-4=0$. The ring of integers is $\Z[\theta]$; the class number is 1;
and a fundamental unit is $\epsilon=-1-\theta+\theta^2$, of norm 1. There are factorizations into
prime ideals as follows:
\[ (2) = p_{21}p_{22}^2=(2-\theta)(-1+\theta)^2, \qquad (5)=p_5^3=(3-\theta+\theta^2)^2. \]
Case I: $c_0=2$.\\
Factoring over $L$ the second equation at (\ref{gheqs}),
\[ (\theta v-u^2) ((\theta^2 - 2\theta + 3)v^2 + (\theta - 2)v u^2 + u^4) = 2 h^2 \]
and the (ideal) gcd of the factors on the left divides $(3-4\theta+3\theta^2)=p_{22}^3 p_5^2$.
Thus
\[ \theta v-u^2 = \pm \epsilon^i p_{22}^j p_5^k \rho^2, \]
for $i,j,k=0,1$, and $\rho \in \Z[\theta]$.  Taking norms
\[ \pm 2^j 5^k = 2 \Box, \]
forcing the plus sign, and $(j,k)=(1,0)$. Hence
\[ \theta v-u^2 = \epsilon^i (-1+\theta) \rho^2, \quad i=0,1. \]
When $i=0$, we have
\[ \theta v-u^2 = (-1+\theta)\rho^2, \qquad v^2+2v u^2-u^4=2g^2, \]
so that
\[ (\theta v-u^2)(v^2+2v u^2-u^4)=2(-1+\theta) \Box, \]
the equation of an elliptic curve of rank 1 over $\Q(\theta)$. The elliptic Chabauty routines in
Magma show that the only points with $u:v$ rational are $(v,\pm u)=(1,0), (1,1)$. \\
When $i=1$,
\[ \theta v-u^2 = \epsilon (-1+\theta)\rho^2, \qquad v^2+2v u^2-u^4=2g^2, \]
so that
\[ (\theta v-u^2)(v^2+2 v u^2 - u^4)=2 \epsilon (-1+\theta) \Box. \]
again, an elliptic curve of rank 1 over $\Q(\theta)$. The only points with $u:v$ rational
are given by $(v,\pm u)=(1,0)$. \\
Case II: $c_0=-1$:\\
Factoring over $K$ the first equation at (\ref{gheqs}),
\[ (v+(1+\sqrt{2})u^2)(v+(1-\sqrt{2})u^2) = -g^2, \]
and the great common divisor of the two factors on the left divides
$2\sqrt{2}$. Thus
\[ v+(1+\sqrt{2})u^2 = \pm e^i (\sqrt{2})^j \alpha^2, \]
for $i,j=0,1$, and $\alpha \in \Z[\sqrt{2}]$. From (\ref{realcond}), we must have the plus
sign.  Taking norms,
\[ (-1)^{i+j}2^j = -\Box, \]
forcing $(i,j)=(1,0)$.
Thus
\[ v+(1+\sqrt{2})u^2 = e \alpha^2 . \]
As above, on factoring over $L$ the second equation at (\ref{gheqs}),
\[ \theta v-u^2 = \pm \epsilon^i p_{22}^j p_5^k \rho^2, \]
where $i,j,k=0,1$, and $\rho \in \Z[\theta]$.
Taking norms,
\[ \pm 2^j 5^k = -\Box, \]
forcing the minus sign, and $(j,k)=(0,0)$. Hence
\[ \theta v-u^2 = -\epsilon^i \rho^2, \qquad (\theta^2-2\theta+3)v^2+(\theta-2)v u^2+u^4 = \epsilon^{-i} \sigma^2, \]
for $i=0,1$, and $\sigma \in \Z[\theta]$.\\
In the case $i=1$, then
\[ \theta v-u^2 = -\epsilon \rho^2, \quad v^2+2v u^2-u^4=-g^2. \]
Thus
\[ (\theta v-u^2)(v^2+2v u^2-u^4)=\epsilon \Box,  \]
and Magma tells us this curve has rank 0 over $L$; the only points are the torsion points
given by $v/u^2=1/\theta$. \\
Suppose finally $i=0$. Trying to work exclusively over the quadratic or
the cubic number field led to problems with the computation. Instead, consider
\[ ((\theta^2-2\theta+3)v^2 + (\theta-2)v u^2 + u^4) (v+(1+\sqrt{2})u^2) = e \Box, \]
the equation of an elliptic curve over the compositum of $K$ and $L$. The rank is determined
to be 2, and the elliptic Chabauty
routines show that the only points with $v:u$ rational are given by
$(v,u^2)=(1,0), (0,1), (-2,1)$.
\end{proof}

\begin{thm}\label{deg9}
The trinomial $x^{9}+A x^2+B$ is divisible by the polynomial
$x^3+u x^2+v x+w$ if and only if
\begin{align*}
&A=u^7-6u^5v+10u^3v^2-4uv^3+5u^4w-12u^2vw+3v^2w+3uw^2,\\
&B=w(u^6-5u^4v+6u^2v^2-v^3+4u^3w-6uvw+w^2),
\end{align*}
and
\begin{equation}
\label{curve9-2}
u^{10}-4u^8v+10u^6v^2-12u^4v^3-3u^2v^4+12v^5=(6(u^2-v)w+u^5-8u^3 v+9 u v^2)^2.
\end{equation}
In particular the only trinomials $x^9+A x^2+B$ with reducibility type $(3,6)$ are those
with the numbers {\rm (6), (7), (8)} on the list in Schinzel~\cite{Sch1}, namely, up to scaling,
\[ x^9 \pm 32x^2 \mp 64, \quad x^9 \pm 81x^2 \mp 54, \quad x^9 \pm 729x^2 \mp 1458. \]
\end{thm}

\begin{proof}
The Division Algorithm is used as in the preceding Theorem to obtain the first statement.
To complete the proof, we need to determine all rational points on the curve
(\ref{curve9-2}) of genus 2. Take the equation in the form
\[ X^5 -3 X^4 z^2 -144 X^3 z^4 +1440 X^2 z^6 -6912 X z^8 + 20736 z^{10} = Y^2, \]
where $X=12 v z^2/u^2$. Set $K=\Q(\theta)$, $\theta^5+\theta^4+4\theta^3+4\theta^2-8\theta+4=0$,
so that
\[ \mbox{Norm}_{K/\Q}(X + (\theta^4+4\theta^2-8)z^2) = Y^2. \]
The ring of integers in $K$ has basis $\{1, \theta, \theta^2, 1/2(\theta^3 + \theta^2), 1/2(\theta^4 + \theta^2)\}$.
There are ideal factorizations:
\[ (2)=p_{21}p_{22}^4, \quad (3)=p_{31}^3p_{32}^2, \quad (7)=p_{71}^3p_{72}^2; \]
and fundamental units are given by
\[ \ep_1=[2,-5,4,-1,2], \quad \ep_2=[0,-11,47,-20,-29].  \]
We have
\[ (X + (\theta^4+4\theta^2-8)z^2) = p_{22}^{i_1} p_{31}^{j_1} p_{32}^{j_2} p_{71}^{k_1} p_{72}^{k_2} \Box, \quad i_1,j_1,j_2,k_1,k_2=0,1; \]
and on taking norms,
\[  2^{i_1} 3^{j_1+j_2} 7^{k_1+k_2} = \Box, \]
so that $i_1=0$, $j_1=j_2$, $k_1=k_2$.
If $k_1=k_2=1$, then $X+4z^2 \equiv 0 \bmod{7}$ and $X+3z^2 \equiv 0 \bmod{7}$, impossible.
If $j_1=j_2=1$, then $3 \mid X$, and $(\theta^4+4\theta^2-8)z^2 \equiv 0 \bmod{p_{31}^3}$, contradiction.
Thus
\[ (X + (\theta^4+4\theta^2-8)z^2) = \Box, \]
so that with $\de=\pm \ep_1^{l_1} \ep_2^{l_2}$, $l_1,l_2=0,1$,
\begin{align*}
& X + (\theta^4+4\theta^2-8)z^2 = \de^{-1} a^2, \\
& X^4 - (\theta^4+4\theta^2-5)X^3z^2 - 12(\theta^3-\theta^2+5\theta+7)X^2z^4 \\
& - 144(\theta^2-2\theta-3)X z^6 - 1728(\theta+1)z^8 = \de b^2,
\end{align*}
for integers $a,b$ of $K$ satisfying $ab=y$.
Eliminating $X$ results in an eighth degree equation homogeneous in $a,z$, that is everywhere
locally solvable for precisely the values $\de=1$, $-\ep_1\ep_2$. \\
Case I: $\de=1$.
\begin{align*}
& X^4 + (-\theta^4-4\theta^2+5)X^3z^2 - 12(\theta^3-\theta^2+5\theta+7)X^2z^4 \\
& - 144(\theta^2-2\theta-3)X z^6 - 1728 (\theta+1)z^8 = b^2.
\end{align*}
The curve is birationally equivalent to
\begin{align*}
E: y^2 = & x^3 - (\theta^4+3\theta^3+6\theta^2+14\theta+6)x^2 + \frac{1}{3}(6\theta^4+19\theta^3+47\theta^2+92\theta+94)x \\
         & + \frac{1}{9}(851\theta^4+1318\theta^3+4040\theta^2+5409\theta-4448),
\end{align*}
of rank 4 over $K$. Generators for a subgroup of odd index in $E(K)/2 E(K)$ are given by
\begin{align*}
(0, \frac{1}{6}(17\theta^4 + 32\theta^3 + 95\theta^2 + 150\theta - 18)); & \\
(\frac{1}{3}(-\theta^4 - 2\theta^3 - 8\theta^2 - 10\theta), \frac{1}{2}(-3\theta^4 - 6\theta^3 - 17\theta^2 - 26\theta + 6)); & \\
(\frac{1}{2}(3\theta^4 + 7\theta^3 + 18\theta^2 + 26\theta), \frac{1}{6}(\theta^4 + 10\theta^3 + 37\theta^2 + 42\theta + 30)); & \\
(\frac{1}{3}(-\theta^4 + \theta^3 - 2\theta^2 + 8\theta + 24), \frac{1}{2}(3\theta^4 + 4\theta^3 + 13\theta^2 + 14\theta - 22)), &
\end{align*}
and the Magma routines succeed in showing the only valid solutions are $(X,z^2)=(1,0),(12,1)$.
(We list the generators above because the initial machine computation returned a subgroup
of index 3, and the routines failed). \\
Case II: $\de=-\ep_1 \ep_2$.
\begin{align*}
& X^4 + (-\theta^4-4\theta^2+5)X^3z^2 - 12(\theta^3-\theta^2+5\theta+7)X^2z^4 \\
&  - 144(\theta^2-2\theta-3)X z^6 - 1728(\theta+1)z^8 = -\ep_1 \ep_2 b^2.
\end{align*}
The point $(0, -48\theta^4-228\theta^3-60\theta^2+264\theta-144)$ leads to birational equivalence
with the curve
\[ y^2 = x^3 + (3\theta^3+\theta^2+4\theta+1)x^2 - (6\theta^4+6\theta^3-20\theta^2+28\theta-35)x - (18\theta^4+45\theta^3+53\theta^2-100\theta-71) \]
of rank 3 over $K$, and the Magma routines are successful in showing that the only
valid solutions arise for $(X,z^2)=(0,1)$.
\end{proof}

Similarly we can obtain the following.
\begin{thm}\label{deg10}
The trinomial $x^{10}+A x+B$ is divisible by the polynomial
$x^3+u x^2+v x+w$ if and only if
\begin{align*}
&A=(w-uv)(u^6-6u^4v+10u^2v^2-4v^3+4u^3w-8uvw+w^2),\\
&B=-w(u^7-6u^5v+10u^3v^2-4uv^3+5u^4w-12u^2vw+3v^2w+3uw^2),
\end{align*}
and
\begin{equation}
\label{curve10-1}
3u^{10}-15u^8v+25u^6v^2-15u^4v^3+3v^5=(3(2u^2-v)w+3u^5-10u^3v+6u v^2)^2
\end{equation}
\end{thm}
\begin{rem}
{\rm Unfortunately, we are unable to determine all rational points
on the curve (\ref{curve10-1}) of genus two.
The rational points $(u,\pm v,w)$ with height at most $10^6$ (with $w \neq0$), and
their corresponding trinomials up to scaling, are as follows:
\begin{equation*}
\begin{array}{ll}
(0,3,3), & x^{10} \pm 297x - 243, \\
(1,1,\frac{2}{3}), &  x^{10} \pm 8019x + 13122, \\
(1,2,\frac{17}{14}), &  x^{10} \pm 261312546880x + 2485545010816.
\end{array}
\end{equation*}
The first example is {\rm (11)} on the list in Schinzel~\cite{Sch1}.
The second and third examples, discovered by Cis{\l}owska, are listed as
{\rm (11a), (12a)} in the reprinting of Schinzel~\cite{Sch1} in Schinzel~\cite{Sch4}.
It is likely they are the only such.
}


\end{rem}

\section{Trinomials with forced factors}
For certain trinomials where $(m,n)>1$ we can force an algebraic factor and determine the
reducibility type of the quotient.

\begin{thm}\label{deg8-2}
Suppose that $x^4+A x+B$ has the rational non-zero root $v$. Then reducibility of
$(x^8+A x^2+B)/(x^2-v)$ implies reducibility type $(3,3)$, occurring precisely when (up to scaling)
\begin{equation}\label{AB-8-2}
A=-\frac{(q^2+2q-1)(9q^2-10q+3)}{4}, \quad B=\frac{(q-1)^2(2q-1)(3q-1)^2}{4}, \; q \neq 1,\frac{1}{2},\frac{1}{3}.
\end{equation}
\end{thm}

\begin{proof}
By Lemma 29 of Schinzel~\cite{Sch1}, if $(x^8+A x^2+B)/(x^2-v)=x^6+v x^4+v^2 x^2+v^3+A$ is reducible,
then it takes the form $(x^3+p x^2+q x+r)(x^3-p x^2+q x-r)$. Comparing coefficients of powers of $x$
and eliminating $r$ gives $(A,B)$ as in the Theorem after scaling so that $p=1$ (it is easy to check
that $p=0$ results in $v=0$). Further, the cubic factor $x^3+x^2+q x-\frac{(q-1)(3q-1)}{2}$ is
irreducible. For if $u$ is a rational root, then $6u^3+7u^2+4u+1=(3q-u-2)^2$. But the corresponding
elliptic curve has rank 0, and the only finite points occur at $u=0,-1/2$, giving $B=0$.
\end{proof}


\begin{thm}
Suppose $x^5+A x+B$ has the quadratic factor $x^2+u x+v$. Then the
polynomial $(x^{10}+A x^2+B)/(x^4+u x^2+v)$ has reducibility type
$(3,3)$ infinitely often, parameterized by the elliptic curve
$Y^2=X(X^2+12X-4)$  of rank 1.
\end{thm}

\begin{proof}
The divisibility condition of the Theorem is that
\[ (A,B)=(-u^4+3u^2 v-v^2, -u v(u^2-2v)), \]
and then
\[ (x^{10}+A x^2+B)/(x^4+u x^2+v) = x^6-u x^4+(u^2-v) x^2+u(2v-u^2). \]
The sextic can only split in the form
\[ x^6-u x^4+(u^2-v) x^2+u(2v-u^2) = (x^3+p x^2+q x+r)(x^3-p x^2+q x-r), \]
and comparing coefficients, $u=p^2-2q$, $v=p^4-4 p^2 q+3 q^2+2 p r$, and
\[ (2q-p^2)(2q^2+4p^2 q-3p^4) = (r+2p(p^2-2q))^2. \]
The latter is the equation of an elliptic curve with model $Y^2=X(X^2+12X-4)$,
having rank 1, and generator $P(X,Y)=(5,3)$. Each multiple of $P$ pulls back
to a factorization of the sextic into two cubics. For the cubics to be reducible,
$x^3+p x^2+q x +r$ will have rational root $w$ say, and necessarily
\[ (2q-p^2)(2q^2+4p^2 q-3p^4) = (-w^3- p w^2-q w+2p(p^2-2q))^2.\]
This latter is the  equation of a curve of genus 4, with only finitely many points
(likely just $(p,q,w)=(\pm 2,2,0)$, $(2,3,-1)$, $(-2,3,1)$). \\
Examples: $(p,q,r)=(2,3,2)$ gives $(u,v)=(-2,3)$, $(A,B)=(11,-12)$ and
\[ x^{10}+11 x^2-12 = (x+1)(x-1)(x^2+x+2)(x^2-x+2)(3-2x^2+x^4). \]
And $(p,q,r)=(2,3,14)$, $(u,v)=(-2,51)$, $(A,B)=(-2005,-9996)$ with
\[ x^{10}-2005 x^2-9996 = (x^3-2x^2+3x-14)(x^3+2x^2+3x+14)(x^4-2x^2+51). \]
\end{proof}

\begin{thm}
Suppose $x^5+A x^2+B$ has the quadratic factor $x^2+u x+v$. Then the
polynomial $(x^{10}+A x^4+B)/(x^4+u x^2+v)$ is reducible (with type
$(3,3)$) in precisely the following four cases (up to scaling):
\begin{align*}
& \bullet x^{10}+6875x^4-312500=(x^3-5x^2+50)(x^3+5x^2-50)(x^4+25x^2+125); \\
& \bullet x^{10}+891 x^4-34992=(x^3-3x^2+9x-18)(x^3+3x^2+9x+18)(x^4-9x^2+108); \\
& \bullet x^{10}-119527785x^4-2195696106864= \\
& (x^3+39x^2+507x+3042)(x^3-39x^2+507x-3042)(x^4+507x^2+237276); \\
& \bullet x^{10}+37347689456x^4-609669805268160000= \\
& (x^3+28x^2+1960x+191100)(x^3-28x^2+1960x-191100)(x^4- 3136x^2+16694496).
\end{align*}
\end{thm}

\begin{proof}
By scaling, we may suppose that $A,B,u,v$ are integers.
Clearly $u \neq 0$, and the divisibility condition is that
\[ A=\frac{u^4-3u^2v+v^2}{u},\qquad B=\frac{-v^2(u^2-v)}{u}, \]
in which case $(x^{10}+A x^4+B)/(x^4+u x^2+v)= x^6-u x^4+(u^2-v)x^2-v(u^2-v)/u$.
The sextic is reducible precisely when
\[ x^6-u x^4+(u^2-v)x^2-v(u^2-v)/u = (x^3+p x^2+q x+r)(x^3-p x^2+q x-r), \]
and comparing coefficients,
\[ p^2-2q-u=0, \quad -q^2+2p r+u^2-v=0, \quad r^2-u v + v^2/u=0. \]
Eliminating $q,r$:
\begin{equation}
\label{pqueqn}
4 p^2 u(-p^4+4p^2 u+u^2)(p^4+3u^2) = (4(4p^2+u) v+ u(p^4-10 p^2 u-3 u^2))^2,
\end{equation}
equivalently,
\begin{equation}
\label{pUeqn}
U(U^2+3p^4)(U^2+12U p^2-9p^4) = V^2, \qquad U=3u,
\end{equation}
where, without loss of generality, $(U,p)=1$.
This latter equation defines a curve of genus 2, and we show that its finite rational points are
precisely \begin{equation}
\label{Upoints}
(U/p^2,\pm V)=(0, 0), (1, 4), (-3,36), (3,36), (-12, 126).
\end{equation}
The first point corresponds to $B=0$, and the remaining points return (up to scaling) the trinomials given in the Theorem. \\
We work in $\Q(\sqrt{5})$, with fundamental unit $\ep=(1+\sqrt{5})/2$. Then
\[ U(U^2+3p^4)(U+3(2+\sqrt{5})p^2)(U+3(2-\sqrt{5})p^2) = \square. \]
Now gcd$(U+3(2+\sqrt{5})p^2, U(U^2+3p^4)(U+3(2-\sqrt{5})p^2))$ divides $2^4 3^3 \sqrt{5}$.
Thus we have:
\[
U+3(2+\sqrt{5})p^2=\gamma u^{-1} \Box, \quad U(U^2+3p^4)(U+3(2-\sqrt{5})p^2))=\gamma u \Box, \]
where the gcd $\gamma$ is a divisor of $2^4 3^3 \sqrt{5}$ and $u$ is a unit. \\ \\
If $\sqrt{5} \mid \gamma$, then $U \equiv -6p^2 \bmod{\sqrt{5}}$,  so $U \equiv -p^2 \bmod{5}$.
It is not possible for both $p,U$ to be divisible by 5; and thus $5 \nmid U$. However,
$U^2+12U p^2-9p^4 \equiv 0 \bmod{5}$, so from (\ref{pUeqn}), necessarily
$U^2+12U p^2-9p^4 \equiv 0 \bmod{25}$. But $(U+6p^2)^2-45p^2 \equiv 0 \bmod{25}$ implies
$p \equiv 0 \equiv U \bmod{5}$, contradiction. \\

Fix the square root of 5 to be positive. If $\gamma u<0$,
then $U+3(2+\sqrt{5})p^2=\gamma u^{-1} \Box$ implies $U < -3(2+\sqrt{5})p^2 < 0$; but then
\[ U<0, \quad U^2+3p^4>0, \quad U+3(2-\sqrt{5})p^2 < -3(2+\sqrt{5}p^2+3(2-\sqrt{5})p^2=-6\sqrt{5}p^2 < 0; \]
so the product of these three terms cannot be negative. \\
Accordingly, $\gamma u$ takes one of the following values: $2^i 3^j \ep^k$, where, without
loss of generality, $i,j,k=0,1$. Of the eight elliptic curves
\[ U(U^2+3p^4)(U+3(2-\sqrt{5})p^2)) = 2^i 3^j \ep^k \Box \]
one has rank 0 (when $(i,j,k)=(0,0,1)$), and the other seven have rank 1. The Magma routines run
satisfactorily to show that the only solutions under the rationality constraint $U/p^2 \in \Q$
are indeed those corresponding to the points at (\ref{Upoints}), together with the point at infinity.
\end{proof}

\begin{thm}
If $x^{4}+A x+B$ has a rational root, say $r$, and the polynomial
$(x^{12}+A x^3+B)/(x^3-r)$ has a cubic factor, then either
\[ (A,B)=(-(r-w)(r^2+w^2), -r w(r^2-r w+w^2)), \qquad r,w \in \Q, \quad r w(r-w) \neq 0, \]
with factorization
\begin{equation*}
x^{12}+A x^3+B = (x^3-r)(x^3+w)(x^6+(r-w)x^3 +r^2-r w+w^2);
\end{equation*}
or, up to scaling, $(A,B)=(128,256), (-5616,-3888)$, with
\begin{align*}
& \bullet x^{12}+128x^3+256=(x^3+4)(x^3+2x^2+4x+4)(x^6-2x^5+8x^2-16x+16), \\
& \bullet x^{12}-5616x^3-3888=(x^3-18)(x^3+6x+6)(x^6-6x^4+12x^3+36x^2-36x+36).
\end{align*}
\end{thm}

\begin{proof}
If $x^4+A x+B$ has rational root $r$, then $B=-A r-r^4$. Suppose that the polynomial
$(x^{12}+A x^3-A r-r^4)/(x^3-r)$ has a cubic factor. Then, say,
\[ x^9+r x^6+r^2 x^3+r^3+A = (x^3+u x^2+v x+w)(x^6+a x^5+b x^4+c x^3+d x^2+e x+f), \]
Comparing coefficients, and eliminating $a,b,c,d,e,f,w$:
\begin{align}
\label{twofac}
(& 3v^3-2r^2+3v r u-2r u^3-3v u^4+u^6)(3v^6-9v^5u^2+2v^3r u^3 +18v^4u^4 \\
& -3v^2r u^5-21v^3u^6+r^2u^6+3v r u^7+15v^2u^8-r u^9-6v u^{10}+u^{12})=0. \nonumber
\end{align}
If the first factor at (\ref{twofac}) is zero then
\begin{equation*}
(4r-u(3v-2u^2))^2=3(8v^3+3v^2u^2-12v u^4+4u^6),
\end{equation*}
the equation of an elliptic curve with minimal model $y^2+x y+y=x^3-x^2-56x+163$. This curve
has rank 0 and and the torsion group is of order three. The torsion points lead to
$(A,B)=(128,256)$. \\
Suppose the second factor at (\ref{twofac}) is zero.
The discriminant in $r$ is
\begin{align*}
&-u^6(8v^6-24v^5 u^2+51v^4 u^4-62v^3 u^6+45v^2 u^8-18v u^{10}+3 u^{12}) \\
& = -u^6 (8 X^6+21 u^4 X^4+6 u^8 X^2+u^{12}/16), \qquad X=v-u^2/2,
\end{align*}
which is negative except when $u=0$, in which case either $v=0$ or $3v^3=2r^2$. The former
leads to the parametrization as stated in the Theorem, the latter to $(A,B)=(-5616,-3888)$.\\
\end{proof}

\begin{thm}
Suppose $x^4+Ax+B$ has a rational root, say $v$. If the polynomial
$(x^{16}+Ax^4+B)/(x^4-v)$ is reducible, then it has reducibility type $(6,6)$,
which occurs when, up to scaling, either (1) $(A,B)$ are given by (\ref{AB-8-2}) in Theorem \ref{deg8-2},
when the sextic factors are cubics in $x^2$; or (2), $(v,A,B)$ is one of the following eight cases:
\begin{align*}
& (72,-347004,-1889568), \quad (-4,1088,4096), \quad(540,-49968576,-58047528960),\\
& (1500,-2975000000,-600000000000), \\
& (1234620,-1767811196564438976,-140874409936505522810880),\\
& (-333000,49083580251562500,4048461902770312500000),\\
& (1506456,-718119113273864316,-4068405448481125418940000) \\
& (3749256176,-52702993391145847275486817276,-29085892289306030859388663640000).
\end{align*}
\end{thm}

\begin{proof}
If the polynomial $x^{16}+A x^4+B$ is divisible by $x^4-v$, then
$B=-A v-v^4$ and
\begin{equation*}
(x^{16}+A x^4+B)/(x^4-v)=x^{12}+v x^8+v^2x^4+v^3+A=f(x^4), \mbox{ say}.
\end{equation*}
Lemma 29 from \cite{Sch1} tells us that either (1)
$f(x^4)=-g(x^2)g(-x^2)$ for a cubic polynomial $g$, which leads to
the values of $(A,B)$ as in Theorem \ref{deg8-2} (and the sextic
$g(x^2)$ is irreducible, as before); or (2), $f(-4x^4)=c
g(x)g(-x)g(i x)g(-i x)$, with $c$ constant and $g(x) \in \Z[x]$ of
degree 3. Thus we obtain
\[ -64  x^{12}+16 v x^8 -4 v^2 x^4+v^3+A = 64 g(x) g(-x) g(i x) g(-i x), \]
where $g(x)=x^3+p x^2+q x+r$, say. This gives
\[ -64 x^{12}+16 v x^8 -4 v^2 x^4+v^3+A = 64 G(x^2) G(-x^2), \]
where $G(X)=X^3+(p^2-2q)X^2+(2p r-q^2)X+r^2$. Equating coefficients of powers
of $X$ and eliminating $v$ gives
\[  2(p^2-2q)(p^8-4p^6 q+4p^4 q^2-3q^4)=(2(7p^2-2q)r+4p^5-16p^3 q+10p q^2)^2. \]
and on setting $X/Z^2=6(2q/p^2-1)$, where $X,Z \in \Z$ and $(X,Z^2)=1$, there results
\[ C: \; X(X^4 + 24 X^3 Z^2 + 24 X^2 Z^4 + 864 X Z^6 +1296 Z^8) = Y^2. \]
We shall show that the set of rational points on $C$ comprises precisely the point at
infinity, with $Z=0$, and the set:
\[ (X/Z^2,\pm Y)=(-18,864), (6,288), (0,0), (-2,32), (1,47), (36, 10152). \]
These points return the trinomials listed in the second statement of the Theorem.\\
Certainly $X=d u^2$ with $d \mid 6$, so $d=\pm 1, \pm 2, \pm 3, \pm 6$. Then
\[ X=d u^2, \quad X^4 + 24 X^3 Z^2 + 24 X^2 Z^4 + 864 X Z^6 +1296 Z^8 = d v^2, \]
and the quartic is locally unsolvable for $d=-1,2,3,-6$. When $d=-3,6$, the quartic
is an elliptic curve of rank 0, and the only solutions are given by $(X,Z^2)=(6,1)$, $(u,v)=(1,48)$.
It remains to deal with $d=1, -2$.\\ \\
Case I: $d=1$:
\[ X=u^2, \qquad X^4 + 24 X^3 Z^2 + 24 X^2 Z^4 + 864 X Z^6 + 1296 Z^8 = v^2, \]
so that
\[ u^8 + 24 u^6 Z^2 + 24 u^4 Z^4 + 864 u^2 Z^6 + 1296 Z^8 = v^2. \]
Over $\Q(\sqrt{3})$,
\[ (u^4 + (12+8\sqrt{3})u^2 Z^2 + 36 Z^4) (u^4 + (12-8\sqrt{3})u^2 Z^2 + 36 Z^4) = v^2; \]
and the gcd of the two factors on the left can be divisible only by $(1+\sqrt{3})$, $(\sqrt{3})$.
Thus
\[ u^4 + (12+8\sqrt{3})u^2 Z^2 + 36 Z^4 = \pm (2+\sqrt{3})^i (1+\sqrt{3})^j (\sqrt{3})^k \Box. \]
Taking norms, $v^2 = (-2)^j (-3)^k \Box$, so that $j=k=0$, and
\[ u^4 + (12+8\sqrt{3})u^2 Z^2 + 36 Z^4 = \de \Box, \quad \de=\pm (2+\sqrt{3})^i. \]
Of the four possibilities for $\de$, only $\de=1$ gives a curve locally solvable above $2$;
and in fact the curve is elliptic with rank $1$. The Magma routines work successfully,
delivering the points $(u,Z)=(1,0),(0,1),(1,1),(6,1)$. \\ \\
Case II: $d=-2$:
\[ X= -2u^2, \qquad X^4 + 24 X^3 Z^2 + 24 X^2 Z^4 + 864 X Z^6 + 1296 Z^8 = -2 v^2, \]
so that
\[ u^8 - 12 u^6 Z^2 + 6 u^4 Z^4 - 108 u^2 Z^6 + 81 Z^8 = -2 V^2. \]
Over $\Q(\sqrt{3})$,
\[ (u^4 + (-6-4\sqrt{3})u^2 Z^2 + 9 Z^4) (u^4 + (-6+4\sqrt{3})u^2 Z^2 + 9 Z^4) = -2 V^2, \]
and the gcd of the two factors on the left can be divisible only by $(1+\sqrt{3})$, $(\sqrt{3})$.
Thus
\[ u^4 + (-6-4\sqrt{3})u^2 Z^2 + 9 Z^4 = \pm (2+\sqrt{3})^i (1+\sqrt{3})^j (\sqrt{3})^k \Box. \]
Taking norms, $-2 V^2 = (-2)^j(-3)^k \Box$, so that $(j,k)=(1,0)$, with
\[ u^4 + (-6-4\sqrt{3})u^2 Z^2 + 9 Z^4 = \de \Box, \qquad \de=\pm (2+\sqrt{3})^i (1+\sqrt{3}). \]
Of the four possibilities for $\de$, only $\de=-(2+\sqrt{3}(1+\sqrt{3})$ gives a curve locally
solvable above $2$, and in fact an elliptic curve of rank $1$. The Magma routines work successfully,
delivering the points $(u,Z)=(1,1),(3,1)$.
\end{proof}

\begin{thm}
The trinomial $x^{15}+Ax^3+B$ is reducible if and only if either $x^5+Ax+B$ is
reducible or up to scaling $(A,B)=(-81,216)$, $(270,729)$ with
\begin{equation*}
\begin{array}{ll}
x^{15}-81x^3+216= & (x^5+3x^4+6x^3+9x^2+9x+6)\times \\
                  & (x^{10}-3x^9+3x^8-6x^5+9x^4-9x^3+27x^2-54x+36), \\
x^{15}+270x^3+729=& (x^5+3x^4+6x^3+9x^2+12x+9)\times \\
                  & (x^{10}-3x^9+3x^8-3x^6+9x^5-18x^4+63x^2-108x+81).
\end{array}
\end{equation*}
\end{thm}
\begin{proof}
If $x^{15}+Ax^3+B$ is reducible and $x^5+Ax+B$ is irreducible, then by
Lemma 29 in \cite{Sch1} we know that
\begin{equation*}
x^{15}+Ax^3+B=f(x)f(\zeta_{3}x)f(\zeta_{3}^2x),
\end{equation*}
where $f(x)=x^5+px^4+qx^3+rx^2+sx+t$.
Expanding the right hand side and equating coefficients:
\[ A=s^3-3r s t+3q t^2, \qquad B=t^3, \qquad p^3-3p q+3r=0, \]
\[ r^3-3q r s+3p s^2+3q^2t-3p r t-3s t=0, \qquad q^3-3p q r+3r^2+3p^2s-3q s-3p t=0. \]
Eliminating $r,s$, noting that $q=p^2$ leads to $B=0$, we get
\begin{align*}
&\left(\frac{(2q-p^2)(2p^6-6p^4q+9p^2q^2-3q^3)-18p q t}{(p^2-q)}\right)^{2}\\
                     &=(p^2-2q)(4p^{10}-24p^8q+60p^6q^2-72p^4q^3+45p^2q^4-18q^5).
\end{align*}
The restriction $p^2=2q$ leads to trivial solutions; and thus
the problem of reducibility of $x^{15}+Ax^3+B$ reduces
to finding all rational points on the genus two curve
\begin{equation*}
C:\;(p^2-2q)(4p^{10}-24p^8q+60p^6q^2-72p^4q^3+45p^2q^4-18q^5)=\square,
\end{equation*}
where we can assume without loss of generality that
$(p^2,q)=1$. We show that the finite points comprise
precisely the following set: $(p^2,q)=(1,0),(1,2),(1,2/3)$, which
lead to the factorizations displayed in the statement of the Theorem.
Set $q/p^2=(X-4Z^2)/(2X)$, $X/Z^2=4p^2/(p^2-2q)$, so that the equation of the
curve takes the form
\[ X^5 + 9 X^4 Z^2 - 48 X^3 Z^4 +864 X^2 Z^6 + 2304 Z^{10} = Y^2. \]
The quintic has precisely one real root $\theta_0$ for $X/Z^2$, $\theta_0 \sim -15.63983$,
and thus $Y^2 \geq 0$ implies $X/Z^2 \geq \theta_0$.
Set $K=\Q(\theta)$ where $\theta^5+2\theta^4+4\theta^3+4\theta^2+2\theta+4=0$. The ring of integers is
$\Z[1,\theta,\theta^2,\theta^3,\theta^4/2]$, and the class number is 1. Fundamental units are given by
\[ \ep_1=1+3\theta^2+\theta^3+\theta^4, \qquad \ep_2=5-3\theta-6\theta^2-4\theta^3-2\theta^4, \]
of norm $+1$. We have the ideal factorizations
\[ (2)=p_2 p_2'^4, \quad (3)=p_3^5, \quad (5)=p_5^2 p_5', \qquad \mbox{Norm}(p_5)=5, \; \mbox{Norm}(p_5')=5^3.  \]
Now
\[ \mbox{Norm}_{K/\Q}(X+ 1/2(5\theta^4 + 8\theta^3 + 20\theta^2 + 12\theta + 10)Z^2) = Y^2, \]
and it may be checked that a prime ideal dividing $X + 1/2(5\theta^4 + 8\theta^3 + 20\theta^2 + 12\theta + 10)Z^2)$
to an odd power must be one of $p_2'=(q_2')=(1-\theta+\theta^2+\theta^4/2)$, $p_3=(q_3)=(-1-\theta)$, $p_5=(q_5)=(1+\theta^2)$.
Thus with $\de=\pm \ep_1^i \ep_2^j q_2'^k q_3^m q_5^n$, $i,j,k,m,n=0,1$, it follows that
\[ X + 1/2(5\theta^4 + 8\theta^3 + 20\theta^2 + 12\theta + 10)Z^2 = \de^{-1} a^2, \]
\begin{align*}
X^4 + & 1/2(-5\theta^4 - 8\theta^3 - 20\theta^2 - 12\theta + 8)X^3Z^2 + (-2\theta^4 - 32\theta^3 - 48\theta^2 - 104\theta - 80)X^2Z^4 \\
 & + (72\theta^4 + 192\theta^3 + 96\theta^2 + 96\theta + 192)XZ^6 + (- 96\theta^4 - 384\theta)Z^8 = \de b^2,
\end{align*}
with integers $a,b$ of $K$ satisfying $q_2'^k q_3^m q_5^n a b = Y$.
Taking norms, $Y^2 = 2^k 3^m 5^n \Box$, so that $k=m=n=0$.
Further, $\ep_1$, $\ep_2$ evaluated at $\theta_0$ are approx. $8.399$ and $0.3477$:
so are positive. Thus $X/Z^2 \geq \theta_0 =-1/2 (5\theta^4 + 8\theta^3 + 20\theta^2 + 12\theta + 10)$ implies the negative sign cannot hold.
Hence in particular
\begin{align*}
X^4 + & 1/2(-5\theta^4 - 8\theta^3 - 20\theta^2 - 12\theta + 8)X^3Z^2 + (-2\theta^4 - 32\theta^3 - 48\theta^2 - 104\theta - 80)X^2Z^4 \\
& + (72\theta^4 + 192\theta^3 + 96\theta^2 + 96\theta + 192)XZ^6 + (- 96\theta^4 - 384\theta)Z^8 = \ep_1^i \ep_2^j b^2,
\end{align*}
with $a b = Y$. The quartic curve is everywhere locally solvable if and only if
$\de=1, \ep_1$. \\
Case I: $\de=1$. Now we have
\begin{align*}
C_1: X^4 - & 1/2(5\theta^4 + 8\theta^3 + 20\theta^2 + 12\theta - 8)X^3Z^2 - (2\theta^4 + 32\theta^3 + 48\theta^2 + 104\theta + 80)X^2Z^4 \\
& + (72\theta^4 + 192\theta^3 + 96\theta^2 + 96\theta + 192)XZ^6 + (- 96\theta^4 - 384\theta)Z^8 = b^2.
\end{align*}
The equation is that of an elliptic curve
of rank 3, and the Magma routines show that the only points on $C_1$
with $X/Z^2 \in \Q$ are given by $(X,Z^2)=(1,0),(0,1)$. \\
Case II: $\de=\ep_1$. Now we have
\begin{align*}
C_2: X^4 + & 1/2(-5\theta^4 - 8\theta^3 - 20\theta^2 - 12\theta + 8)X^3Z^2 + (-2\theta^4 - 32\theta^3 - 48\theta^2 - 104\theta - 80)X^2Z^4 \\
& + (72\theta^4 + 192\theta^3 + 96\theta^2 + 96\theta + 192)X Z^6 + (- 96\theta^4 - 384\theta)Z^8 = \ep_1 b^2,
\end{align*}
with point at $(X,b,Z^2)=(4, -32\theta^3-32\theta^2-32\theta-64, 1)$.
The curve is elliptic with $K$-rank 3; and the Magma routines work satisfactorily to show that the only points on $C_2$
with $X/Z^2 \in \Q$ are $(X,Z^2)=(4,1),(-12,1)$.
\end{proof}

\section{Concluding remarks and some new sporadic trinomials}
Generalizing a statement in Theorem \ref{1111type}, we have the following result.
\begin{thm}\label{111n}
There are no trinomials $x^n+A x+B$, $n$ even, with three linear factors.
\end{thm}
\begin{proof}
By considering the four possibilities for sign of $A,B$, it follows immediately from
Descartes' Rule of Signs that the polynomial $x^n+A x+B$, $n$ even, can have at most two real
roots; and the assertion of the Theorem follows.
\end{proof}

\begin{rem}
{\rm This argument applies over any real field, so that
reducibility type $(1,1,1,n-3)$ is impossible for trinomials $x^n+A x+B$, $n$ even, defined over any real field.}
\end{rem}

\noindent
When $n$ is {\it odd}, let $p,q,r$ be distinct rational roots of $x^n+A x+B$ (where now Descartes' Rule
of Sign implies $A<0$).
Then
\[ p^n+A p+B=0, \quad q^n+A q+B=0, \quad r^n+A r+B=0, \]
and eliminating $A,B$,
\[ (q-r)p^n+(r-p)q^n+(p-q)r^n= (p-q)(q-r)(r-p) G_n(p,q,r) = 0. \]
The equation $G_n(p,q,r)$ (which represents a curve of genus $(n-4)(n-3)/2$)
has indeed (trivial) points, but we do not know how to show these trivial points form the
complete set of solutions, which we believe to be the case. Indeed, we firmly believe that
the following conjecture is true.

\begin{conj}\label{111n-3typeconj}
Let $n\geq 4$. There are no trinomials $x^n+A x+B$ defined over $\Q$ with reducibility type $(1,1,1,n-3)$.
\end{conj}

Finally, while studying the trinomials of this paper, the following sporadic trinomial factorizations came to light,
and do not appear to have been previously recorded:

\begin{center}
\begin{tabular}{c|c}
trinomial & factor \\ \hline
$x^9 + 27 x^4 - 108$ & $x^3 + 3 x^2 + 6 x+ 6$ \\
$x^{10} + 297 x^3 + 648$ & $x^5 + 3 x^3 + 9 x^2 - 9 x + 18$ \\
$x^{11}+12x+8$ & $x^5 - 2x^4 + 2x^3 - 2x^2 + 2$ \\
$x^{11} - 6184976 x^3 + 4216540160$ &  $x^3 + 14 x^2 + 98 x + 392$ \\
$x^{13} - 340224 x + 732160$ & $x^3 - 2 x^2 + 8 x - 20$ \\
$x^{16} + 3486328125 x + 9277343750$ & $x^3 + 5 x^2 + 25 x + 50$ \\
$x^{16} + 34816 x^3 - 552960$ & $x^4 + 2 x^3 - 8 x - 24$. \\ \hline
\end{tabular}
\end{center}

\vspace{0.25in}

\noindent
{\bf Acknowledgements}: We are grateful to the referee for a careful reading of the original
manuscript, and for helpful suggestions. All computations in this paper were carried out
using Magma~\cite{Mag}. The research of the second author was supported by Polish Government funds for science, grant I 2010 044 770 for the years 2010--2011.

\vspace*{0.25in}

\noindent
School of Mathematics and Statistical Sciences, Arizona State University, Tempe AZ 85287-1804, USA; e-mail: bremner@asu.edu, \\ \\
Jagiellonian University, Faculty of Mathematics and Computer Science, Institute of Mathematics, {\L}ojasiewicza 6, 30-348 Krak\'ow, Poland; email: Maciej.Ulas@im.uj.edu.pl
\end{document}